\documentclass[12pt]{amsart}
\usepackage{amssymb,slashed, color,array}
\usepackage{amsmath,amsfonts,amsthm,euscript,mathrsfs,multirow}
\usepackage[all]{xy}
\usepackage[english]{babel}
\usepackage{epsf}
\usepackage{graphicx}
\csname @addtoreset\endcsname{equation}{section}

\newtheorem{theorem}{Theorem}[section]

\newtheorem{corol}[theorem]{Corollary}

\newtheorem{question}{Question}

\theoremstyle{definition}
\newtheorem{definition}[theorem]{Definition}
\newtheorem{remark}{Remark}[section]
\newtheorem{example}[theorem]{Example}

\newenvironment{notation and conventions}{\textbf{Notation and conventions.}}{ }
\DeclareFontFamily{U}{rsf}{} \DeclareFontShape{U}{rsf}{m}{n}{ <5> <6> rsfs5 <7> <8> <9> rsfs7 <10-> rsfs10}{}
\DeclareMathAlphabet\Scr{U}{rsf}{m}{n}
\definecolor{pink}{rgb}{1,0,1}

  \usepackage[pdftex]{hyperref}

\begin{document}

\begin{center}
\baselineskip=14pt{\LARGE
 On Euler polynomials for projective hypersurfaces\\
}
\vspace{1.5cm}
{\large  James Fullwood$^{\spadesuit}$
  } \\
\vspace{.6 cm}

${}^\spadesuit$Institute of Mathematical Research, The University of Hong Kong, Pok Fu Lam Road, Hong Kong.\\

\end{center}

\vspace{1cm}
\begin{center}

{\bf Abstract}
\vspace{.3 cm}
\end{center}

{\small
For every positive integer $n\in \mathbb{Z}_+$ we define an `Euler polynomial' $\mathscr{E}_n(t)\in \mathbb{Z}[t]$, and observe that for a fixed $n$ all Chern numbers (as well as other numerical invariants) of all smooth hypersurfaces in $\mathbb{P}^n$ may be recovered from the single polynomial $\mathscr{E}_n(t)$. More generally, we show that all Chern classes of hypersurfaces in a smooth variety may be recovered from its top Chern class.}








\vspace{0.25in}

Fix an algebraically closed field $\mathfrak{K}$ of characteristic zero. We denote by $\mathbb{P}^n$ projective $n$-space over $\mathfrak{K}$ and for a smooth subvariety $X$ of $\mathbb{P}^n$ we define its Euler characteristic denoted $\chi(X)$ to be $\int_X c(TX)\cap [X]$ \footnote{For $\mathfrak{K}=\mathbb{C}$ certainly $\chi(X)=\chi_{\text{top}}(X)$.}. For every positive integer $n\in \mathbb{Z}_+$, there exists a polynomial $\mathscr{E}_n(t)\in \mathbb{Z}[t]$ such that if $X\subset \mathbb{P}^n$ is a smooth hypersurface of degree $d$, then

\[
\chi(X)=\mathscr{E}_n(d).
\]
\\
The adjunction formula along with standard exact sequences yield

\[
\mathscr{E}_n(t)=(-1)\sum_{k=0}^{n-1}\left(\begin{array}{c} n+1 \\ k \end{array}\right)(-t)^{n-k},
\]
\\
which we refer to as the \emph{n\emph{th} Euler polynomial}. In this note we make the observation that not just the Euler characteristic, but \emph{all} Chern numbers of \emph{all} smooth projective hypersurfaces in $\mathbb{P}^n$ (along with all Euler characteristics of all their general hyperplane sections) may be recovered from $\mathscr{E}_n(t)$ (for all $n$).

\begin{remark}
More generally, for $X\subset \mathbb{P}^n$ a (possibly) \emph{singular} hypersurface of degree $d$ over $\mathbb{C}$, it is known that

\begin{equation}
\chi(X)=\mathscr{E}_n(d)+\int_X \mathcal{M}(X),
\end{equation}
\\
where $\chi(X)$ here denotes topological Euler characteristic with compact supports, and $\mathcal{M}(X)$ denotes the \emph{Milnor class}\footnote{For more on Milnor classes from an algebraic perspective see \cite{FMC}.} of $X$, a characteristic class supported on the singular locus of $X$ (so that $\mathcal{M}(X)=0$ for $X$ smooth). As $X$ is in the same rational equivalence class as a smooth hypersurface of degree $d$, the Milnor class then measures the deviation of $\chi(X)$ from that of a smooth deformation (parametrized by $\mathbb{P}^1$). As the Milnor class is defined for any $\mathfrak{K}$-variety, in this more general setting we \emph{define} the Euler characteristic of a possibly singular hypersurface $X$ to be $\chi(X):=\int_X(T_{\text{vir}}(X)+\mathcal{M}(X))$, where $T_{\text{vir}}(X)$ denotes the virtual tangent bundle of $X$. Thus formula (0.1) holds without the assumption $\mathfrak{K}=\mathbb{C}$.  
\end{remark}

So let $R$ be a ring and let $\vartheta:R[t]\to R[t]$ be the map which keeps only the terms of degree greater than one, and then divides the result by $-t$, i.e., the map given by

\[
a_nt^n+\cdots +a_0 \mapsto -(a_nt^{n-1}+\cdots +a_2t).
\]
Our result is the following

\begin{theorem}\label{mt} Let $n\in \mathbb{Z}_+$ and

\[
\mathfrak{C}_n(s,t)=\vartheta^{n-1}\mathscr{E}_n(t)s +\vartheta^{n-2}\mathscr{E}_n(t)s^2+\cdots + \vartheta\mathscr{E}_n(t)s^{n-1}+\mathscr{E}_n(t)s^n\in \mathbb{Z}[s,t],
\]
\\
where $\vartheta^k$ denotes the $k$-fold composition of the map $\vartheta:\mathbb{Z}[t]\to \mathbb{Z}[t]$. Then

\[
\mathfrak{C}_n(H,d)=\iota_*c(TX)\cap [X],
\]
\\
where $X\subset \mathbb{P}^n$ is a smooth hypersurface of degree $d$, $H=c_1(\mathscr{O}_{\mathbb{P}^n}(1))\cap [\mathbb{P}^n]$ and $\iota:X\hookrightarrow \mathbb{P}^n$ is the natural inclusion.
\end{theorem}
We then immediately arrive at the following

\begin{corol}\label{cn} Let $X$ be a hypersurface of degree $d$ in $\mathbb{P}^n$. Then all Chern numbers of $X$ are of the form

\[
\prod_{i}\mathscr{E}_{n}^{j_i}(d),
\]
\\
where $\sum j_i=n-1$ and $\mathscr{E}_n^k(t)$ denotes $\vartheta^{n-(k+1)}\mathscr{E}_n(t)$.
\end{corol}

Theorem \ref{mt} actually follows from a more general fact about hypersurfaces, which is nothing more than an elementary observation about the \emph{Fulton class} of a hypersurface. Before stating the result we need the following

\begin{definition} Let $M$ be a smooth $\mathfrak{K}$-variety and $Y\hookrightarrow M$ be a regular embedding. The \emph{Segre class} of $Y$ relative to $M$ is then given by

\[
s(Y,M):= c(N_YM)^{-1}\cap [Y] \in A_*Y,
\]
where $N_YM$ denotes the normal bundle to $Y$ in $M$. The \emph{Fulton class} of $Y$ is then given by
\[
c_{\text{F}}(Y):=c(TM)\cap s(Y,M).
\]
\end{definition}

\begin{remark}
In \cite{IntersectionTheory} (Chapter 4), Fulton actually defines the relative Segre class for $Y$ an arbitrary \emph{subscheme} of $M$, and proves the class $c(TM)\cap s(Y,M)$ is intrinsic to $Y$ (i.e., it is independent of its embedding into a smooth variety). However, as we consider only hypersurfaces in this note (which are \emph{always} regularly embedded), we don't need the definition in its full generality. In any case, for $Y$ smooth $c_{\text{F}}(Y)$ coincides with its usual Chern class.
\end{remark}

The previous results stated in this note are consequences of the following

\begin{theorem}\label{mg} Let $M$ be a smooth $\mathfrak{K}$-variety of dimension $n$, denote its Chern classes by $c_i$ and let $\mathcal{E}_n(s)\in A_*M[s]$ be given by

\[
\mathcal{E}_n(s)=c_{n-1}s-c_{n-2}s^2+\cdots +(-1)^nc_1s^{n-1}+(-1)^{n+1}s^n.
\]
\\
Then if $X$ is any hypersurface in $M$ we have

\begin{equation}
c_{\emph{F}}(X)=\vartheta^{n-1}\mathcal{E}_n(X) +\vartheta^{n-2}\mathcal{E}_n(X)+\cdots + \vartheta\mathcal{E}_n(X)+\mathcal{E}_n(X),
\end{equation}
\\
where $c_{\emph{F}}(X)$ denotes the `Fulton class' of $X$, and $X$ on the RHS of formula \emph{(0.2)} denotes its class in $A_*M$. In particular, all Fulton classes of $X$ may be recovered from its top Fulton class via the map $\vartheta$.
\end{theorem}

\begin{proof}
Let $X$ be a (possibly singular) hypersurface in a smooth variety $M$ and denote its class in $A_*M$ simply by $X$. Then its Fulton class is defined as $c(TM)\cap s(X,M)$, where $s(X,M)$ denotes the Segre class of $X$ in $M$. Since $X$ is a hypersurface, it is regularly embedded, thus $s(X,M)=\frac{X}{1+X}$. Its Fulton classes (i.e., Chern classes for $X$ smooth) then take the following form (as a class in the Chow group $A_*M$):

\begin{eqnarray*}
c_0(X)&=&X \\
c_1(X)&=&c_1X-X^2 \\
c_2(X)&=&c_2X-c_1X^2+X^3 \\
\vdots \\
c_{n-1}(X)&=&c_{n-1}X-c_{n-2}X^2+\cdots +(-1)^nc_1X^{n-1}+(-1)^{n+1}X^n=\mathcal{E}_n(X),\\
\end{eqnarray*}
where by $X^k$ we mean the $k$-fold intersection product of $X$ with itself. The theorem immediately follows.
\end{proof}

Theorem \ref{mt} may then be obtained by replacing $M$ by $\mathbb{P}^n$ and $X$ by $tc_{1}(\mathscr{O}_{\mathbb{P}^n}(1))\cap [\mathbb{P}^n]$.

\begin{remark} Not only the Chern numbers, but the Euler characteristics of all general hyperplane sections of all hypersurfaces in $\mathbb{P}^n$ may be easily recovered from $\mathscr{E}_n(t)$ as well. More precisely, let

\[
\mathfrak{C}_n^{\vee}(s,t):=s^n \mathfrak{C}_n\left(\frac{1}{s},t\right)=\mathscr{E}_n(t)+\vartheta\mathscr{E}_n(t)s+\cdots +\vartheta^{n-1}\mathscr{E}_n(t)s^{n-1},
\]
where $\mathfrak{C}_n(s,t)$ is as given in the statement of Theorem \ref{mt} (i.e., the power of $s$ is now keeping track of dimension rather than codimension), and let

\[
\mathfrak{e}_n(s,t):=\frac{s\cdot \mathfrak{C}_n^{\vee}(-s-1,t)+\mathfrak{C}_n^{\vee}(0,t) }{s+1}.
\]
Then by Theorem 1.1 in \cite{AluffiEuler}, if $X\subset \mathbb{P}^n$ is a smooth hypersurface of degree $d$ then the coefficient of $(-s)^r$ in $\mathfrak{e}_n(s,d)$ is the Euler characteristic of $X\cap H_1\cap \cdots \cap H_r$, where the $H_i$ are general hyperplanes with respect to $X\cap H_1\cap \cdots \cap H_{i-1}$.
\end{remark}

\begin{remark} For $X$ a smooth $\mathfrak{K}$-variety it was shown in \cite{FvHH} that the (unnormalized) \emph{motivic Hirzebruch class} of $X$\footnote{Motivic Hirzebruch classes were first defined in \cite{MHC}. For a pedagogic introduction to motivic characteristic classes we recommend \cite{MCC}.} (referred to as the ``Hirzebruch series" in [loc. cit.]), denoted $T_y^*(X)$, may be given by 

\begin{equation}
T_y^*(X)=\frac{(1+y)^k}{k!}\left.\frac{d^k}{ds^k}\exp\left(\ln\left(\frac{s(1+ye^{-s})}{1-e^{-s}}\right) \odot \frac{-sC'}{C} \right)\right|_{s=0},
\end{equation}
\\
where $k=\text{dim}(X)$, $C=1-c_1(X)s+\cdots +(-1)^kc_k(X)s^k$, $C'=\frac{d}{ds}C$, and $\odot$ denotes the \emph{Hadamard product} of power series\footnote{For $f=\sum a_i s^i$ and $g=\sum b_is^i$, then $f\odot g= \sum a_ib_i s^i$.}. Thus for $X\subset \mathbb{P}^n$ a smooth hypersurface of degree $d$, replacing $C$ by $\mathfrak{C}_n(-sH,d)$ in formula (0.3) yields a formula for the motivic Hirzebruch class of $X$ in terms of $\mathscr{E}_n(t)$, where $\mathfrak{C}_n$ and $H$ are as given in Theorem~\ref{mt}. The degree zero part of $T_y^*(X)$ then yields the Euler characteristic, arithmetic genus and signature of $X$ when evaluated at $y=-1,0,1$, respectively. For $\mathfrak{K}=\mathbb{C}$ and $X$ hyperk\"ahler, a geometric interpretation for arbitrary $y$ is given in \cite{GTH}.
\end{remark}

\begin{example}
Let $n=4$. Then

\begin{eqnarray*}
\mathscr{E}_4(t)&=&10t-10t^2+5t^3-t^4 \\
\vartheta\mathscr{E}_4(t)&=&10t-5t^2-t^3 \\
\vartheta^2\mathscr{E}_4(t)&=&5t-t^2. \\
\end{eqnarray*}
Thus for $X\subset \mathbb{P}^4$ a smooth hypersurface of degree $d$ by Corollary \ref{cn} we have

\begin{eqnarray*}
c_1(X)^3&=&(\vartheta^2\mathscr{E}_4(d))^3=(5d-d^2)^3 \\
c_1(X)c_2(X)&=&\vartheta^2\mathscr{E}_4(d)\cdot \vartheta\mathscr{E}_4(d)=(5d-d^2)(10d-5d^2-d^3) \\
c_3(X)&=&\mathscr{E}_4(d)=10d-10d^2+5d^3-d^4.
\end{eqnarray*}
Moreover, the Lefschetz hyperplane theorem and Hirzebruch-Riemann-Roch yields

\begin{eqnarray*}
h^{0,3}(X)&=&1-\frac{\vartheta^2\mathscr{E}_4(d)\cdot \vartheta\mathscr{E}_4(d)}{24}, \\
h^{1,2}(X)&=&\frac{\vartheta^2\mathscr{E}_4(d)\cdot \vartheta\mathscr{E}_4(d)}{24}-\frac{\mathscr{E}_4(d)-2}{2},
\end{eqnarray*}
which are the only nontrivial Hodge numbers of $X$. Furthermore, we have

\begin{eqnarray*}
\mathfrak{e}_4(s,d)&=&\mathscr{E}_4(d)+(-\vartheta\mathscr{E}_4(d)+\vartheta^2\mathscr{E}_4(d)-\vartheta^3\mathscr{E}_4(d))s \\
                   & &+(\vartheta^2\mathscr{E}_4(d)-2\vartheta^3\mathscr{E}_4(d))s^2+\vartheta^4\mathscr{E}_4(d)s^3  \\
                   &=&(10d-10d^2+5d^3-d^4)+(-6d+4d^2-d^3)s+(3d-d^2)s^2-ds^3,
\end{eqnarray*}
illustrating the fact that all Chern numbers, all Hodge numbers and all Euler characteristics of general hyperplane sections of all smooth hypersurfaces in $\mathbb{P}^4$ may be obtained via $\mathscr{E}_4(t)$ and the map $\vartheta$.
\end{example}

We end with the following 

\begin{question} 
Is there a more general class of varieties (besides hypersurfaces) for which all of its Chern classes may be recovered from its top Chern class?
\end{question}

\emph{Acknowledgements}. We thank Paolo Aluffi for pointing us in the more general direction of Theorem \ref{mg}, from which the proof of Theorem \ref{mt} immediately follows.

\bibliographystyle{plain}
\bibliography{mybib}

\begin{thebibliography}{1}

\bibitem{AluffiEuler}
P.~Aluffi.
\newblock Euler characteristics of general linear sections and polynomial
  {C}hern classes.
\newblock {\em Rend. Circ. Mat. Palermo (2)}, 62(1):3--26, 2013.

\bibitem{MHC}
Jean-Paul Brasselet, J{\"o}rg Sch{\"u}rmann, and Shoji Yokura.
\newblock Hirzebruch classes and motivic {C}hern classes for singular spaces.
\newblock {\em J. Topol. Anal.}, 2(1):1--55, 2010.

\bibitem{FMC}
J.~Fullwood.
\newblock On {M}ilnor classes via invaraints of singular schemes.
\newblock {\em Journal of singularities}, 8:1--10, 2014.

\bibitem{FvHH}
J.~Fullwood and M.~van Hoeij.
\newblock On {H}irzebruch invariants of elliptic fibrations.
\newblock In {\em String-{M}ath 2011}, volume~85 of {\em Proc. Sympos. Pure
  Math.}, pages 355--366. Amer. Math. Soc., Providence, RI, 2012.

\bibitem{IntersectionTheory}
William Fulton.
\newblock {\em Intersection theory}, volume~2 of {\em Ergebnisse der Mathematik
  und ihrer Grenzgebiete. 3. Folge.}
\newblock Springer-Verlag, Berlin, second edition, 1998.

\bibitem{GTH}
G.~Thompson.
\newblock A geometric interpretation of the {$\chi_y$} genus on
  hyper-{K}\"ahler manifolds.
\newblock {\em Comm. Math. Phys.}, 212(3):649--652, 2000.

\bibitem{MCC}
Shoji Yokura.
\newblock Motivic characteristic classes.
\newblock In {\em Topology of stratified spaces}, volume~58 of {\em Math. Sci.
  Res. Inst. Publ.}, pages 375--418. Cambridge Univ. Press, Cambridge, 2011.

\end{thebibliography}

\end{document}